\documentclass[12pt]{article}
\usepackage{amsmath}     
\usepackage{amssymb}     
\usepackage{amsthm}
\usepackage{amsfonts}
\usepackage{comment}
\usepackage{hyperref}
\usepackage[font=small]{caption}
\usepackage{tikz}
\usepackage{rotating}
\usepackage[showonlyrefs]{mathtools}
\mathtoolsset{showonlyrefs}
\numberwithin{equation}{section}
\def\A{\operatorname{av}}
\def\Ai{\operatorname{Ai}}
\def\re{\operatorname{Re}}
\def\im{\operatorname{Im}}
\def\arg{\operatorname{arg}}
\def\N{\mathbb{N}}
\def\Z{\mathbb{Z}}
\def\R{\mathbb{R}}
\def\C{\mathbb{C}}
\def\CC{\widehat{\mathbb{C}}}
\def\D{\mathbb{D}}
\def\F{\mathcal{F}}
\def\O{\mathcal{O}}
\newtheorem{lemma}{Lemma}[section]
\newtheorem{theorem}{Theorem}[section]
\newtheorem{proposition}{Proposition}[section]
\newtheorem{thmx}{Theorem}

\theoremstyle{remark}
\newtheorem{remark}{Remark}[section]
\newtheorem{example}{Example}[section]
\newtheorem*{ack}{Acknowledgment}
\begin{document}  
\title{Meromorphic functions with three radially distributed values}
\author{Walter Bergweiler and Alexandre Eremenko\thanks{Supported by NSF grant DMS-1665115}}
\date{}
\maketitle
\centerline{\emph{Dedicated to the memory of Walter K.\ Hayman}}
\begin{abstract}
We consider transcendental meromorphic functions for which the zeros, $1$-points and
poles are distributed on three distinct rays.
We show that such functions exist if and only if the rays are equally spaced.
We also obtain a normal family analogue of this result.

\medskip


\medskip

\end{abstract}
\section{Introduction and results}\label{intro}
Our starting point is the following result.
\begin{thmx} \label{thm-bier-mill}
There is no transcendental entire function for which all zeros lie on one ray and
all $1$-points lie on a different ray.
\end{thmx}
This was proved by Biernacki \cite[p.~533]{Biernacki1929} and Milloux~\cite{Milloux1927}
for functions of finite order; see also~\cite{BEH}.
The restriction on the order can be omitted
by a later result of  Edrei~\cite{Edrei1955}. This result yields that
if all zeros and $1$-points of an entire function $f$ lie on finitely many rays,
then $f$ has finite order.

The following normal family analogue of Theorem~\ref{thm-bier-mill} was proved 
in~\cite[Theorem~1.1]{BE2019}.
Here $\D$ denotes the unit disk.
\begin{thmx}\label{thm-norm}
Let $L_0$ and $L_1$ be two distinct rays emanating from the origin and
let $\F$ be the family of all functions holomorphic in $\D$ for which all zeros
lie on $L_0$ and all $1$-points lie on $L_1$.
Then $\F$ is normal in $\D\backslash\{0\}$.
\end{thmx}
The purpose of this paper is to consider analogues of these results for meromorphic
functions, with poles being distributed on some further ray. First we note that 
there exist meromorphic functions for which zeros, $1$-points and poles lie on
three distinct rays.
Such a function is given by the following example.
Recall here that the Airy function $\Ai$ is an entire function which satisfies the 
differential equation $\Ai''(z)=z\Ai(z)$; see, e.g.,~\cite[\S 9.2]{DLMF}.
\begin{example}\label{example1}
Let
\begin{equation} \label{ex1}
f(z)=e^{\pi i/3}\frac{\Ai(e^{2\pi i/3}z)}{\Ai(e^{-2\pi i/3}z)} .
\end{equation}
Then all zeros of $f$ are on the ray $\{re^{i\pi/3}\colon r>0\}$, all poles of $f$ are on
$\{re^{-i\pi/3}\colon r>0\}$ and all $1$-points of $f$ are on the negative real axis.
\end{example}
We will verify at the beginning of Section~\ref{sec-proof1} that the function $f$ defined 
in Example~\ref{example1} has the properties stated there.

We note that in Example~\ref{example1} the rays are equally spaced.
If the rays are not equally spaced, then we have analogues of Theorems~\ref{thm-bier-mill}
and~\ref{thm-norm}.
\begin{theorem}\label{theorem2}
Let $L_0$, $L_1$ and $L_\infty$ be three distinct rays emanating from the origin. 
If the rays are not equally spaced, then 
there is no transcendental meromorphic 
function for which all but finitely many zeros lie on $L_0$, all but finitely 
many $1$-points lie on $L_1$ and all but finitely many poles lie on $L_\infty$.
\end{theorem}
\begin{theorem}\label{theorem1}
Let $L_0$, $L_1$ and $L_\infty$ be three distinct rays emanating from the origin and let 
$0\leq r<R\leq \infty$. Let
$\F$ be the family of all functions meromorphic in $\{z\in\C\colon r<|z|<R\}$ for which 
all zeros are on $L_0$, all $1$-points are on $L_1$ and all poles are on~$L_\infty$.

Then $\F$ is normal if and only if the rays are not equally spaced.
\end{theorem}
One can deduce Theorem~\ref{thm-bier-mill} from Theorem~\ref{thm-norm} by
considering the family $\{f(rz)\colon r>0\}$.
Given any transcendental entire function~$f$, 
this family is not normal at some point in $\C\setminus\{0\}$.
This approach does not suffice to deduce 
Theorem~\ref{theorem2} from Theorem~\ref{theorem1}, 
since there are meromorphic functions $f$ for which the family $\{f(rz)\colon r>0\}$
is normal in $\C\setminus\{0\}$.
Such functions were called \emph{Julia exceptional functions} by
Ostrowski~\cite[Kapitel II]{Ostrowski1925} who studied them in detail. 
They are also called \emph{normal functions}.
Lehto and Virtanen~\cite{Lehto1957} introduced this terminology for functions meromorphic in
a domain~$G$, but in  the case that $G=\C\setminus\{0\}$ it reduces to the 
property stated above; 
see also~\cite{Eremenko2007,Eremenko2010} for a discussion of normal functions.

The differential equation satisfied by the Airy function 
implies that the function $f$ given by~\eqref{ex1}
satisfies the differential equation
\begin{equation} \label{i1}
S(f)(z)=-2z,
\end{equation}
where 
\begin{equation} \label{i2}
S(f)=\left(\frac{f''}{f'}\right)'- \frac12 \left(\frac{f''}{f'}\right)^2
=\left(\frac{f'''}{f'}\right)'- \frac32 \left(\frac{f''}{f'}\right)^2
\end{equation}
denotes the Schwarzian derivative.

The following result says that  -- in some sense --
all meromorphic functions for which zeros, $1$-points and poles are distributed 
on three rays are related to the function of Example~\ref{example1}.
\begin{theorem}\label{theorem3}
Let $L^1$, $L^2$ and $L^3$ be three equally spaced rays and let $f$ be a
transcendental meromorphic function. Then there exist distinct values  $a_1$, $a_2$
and $a_3$ such that all but finitely many $a_j$-points are on $L^j$ if
and only if 
\begin{equation} \label{i3}
S(f)(z)=e^{3\theta i} z R(z^3),
\end{equation}
where $\theta$ is the argument of one of the rays $L^j$ and $R$ is a real
rational function satisfying $0 < R(\infty) < \infty$.
\end{theorem}
If $L$ is a linear fractional transformation, then $S(L\circ f)=S(f)$.
One may choose $L$ such that $a_1$, $a_2$ and $a_3$ are mapped to $0$, $1$ and~$\infty$.

The rational functions $Q$ for which the equation $S(f)=Q$ has a meromorphic solution $f$ 
have been classified by Elfving~\cite[Kapitel IV]{Elfving1934};
see also~\cite{Cui2016}, \cite[Theorem~ 6.7]{Laine1993} and~\cite{Laine1986}.

An example of 
a rational function $R$ with poles such that~\eqref{i3} has a solution $f$
for which all (and not only all but finitely many) zeros, $1$-points and poles are on the rays
will be given in Remark~\ref{Rpoles}.

Nevanlinna~\cite{Nevanlinna1966} raised the following interpolation problem:
Given points $c_1,\dots,c_q$ on the Riemann sphere and $q$ sequences 
$(z_{1,k})_{k\in\N},\dots, (z_{q,k})_{k\in\N}$ in $\C$,
when does there exists a meromorphic function $f$ such that the $c_j$-values 
are precisely the points  $z_{j,k}$?
For $q=2$ such a function exists by the Weierstra{\ss} factorization theorem,
so the problem addresses the case that $q\geq 3$.
Theorems~\ref{theorem2} and~\ref{theorem3} may also be considered as a 
contribution to this problem of Nevanlinna.

\begin{ack}
We thank Fedor Nazarov for suggesting the proof of Proposition~\ref{proposition1}.
We also thank four referees and the editor, David Drasin, for many helpful suggestions.
\end{ack}

\section{Proof of Theorem~\ref{theorem1}}\label{sec-proof1}
As we will use Example~\ref{ex1} in one direction of the proof, we begin
by verifying the properties of this example.
\begin{proof}[Verification of Example~\ref{ex1}]
The zeros of the Airy function are all negative~\cite[\S 9.9]{DLMF}.
This implies that all zeros of $f$ are on $\{re^{i\pi/3}\colon r>0\}$
and  all poles of $f$ are on $\{re^{-i\pi/3}\colon r>0\}$.
By~\cite[Equation 9.2.12]{DLMF} we have
\begin{equation} \label{a1}
\Ai(z)+e^{-2\pi i/3} \Ai(e^{-2\pi i/3}z) +e^{2\pi i/3} \Ai(e^{2\pi i/3}z)=0.
\end{equation}
This implies that 
\begin{equation} \label{a1a}
\begin{aligned}
f(z)-1
&=\frac{e^{\pi i/3}\Ai(e^{2\pi i/3}z)-\Ai(e^{-2\pi i/3}z)}{\Ai(e^{-2\pi i/3}z)}
\\ &
=e^{-\pi i/3}\frac{e^{2 \pi i/3}\Ai(e^{2\pi i/3}z)+e^{-2\pi i/3}\Ai(e^{-2\pi i/3}z)}{\Ai(e^{-2\pi i/3}z)}
\\ &
=-e^{-\pi i/3}\frac{\Ai(z)}{\Ai(e^{-2\pi i/3}z)}.
\end{aligned}
\end{equation}
Hence the $1$-points of $f$ are all negative.
\end{proof}

Let $\overline{D}(a,r)$ denote the closed disk of radius $r$ around a point~$a$.
The following result was proved in~\cite[Theorem~1.3]{BE2019} 
and~\cite[Proposition~1.1]{BE2020}.
\begin{lemma}\label{lemma2}
Let $D$ be a domain and let $L$ be a straight line which divides
$D$ into two subdomains $D^+$ and $D^-$.
Let $\F$ be a family of functions holomorphic in $D$ which do not have zeros
in $D$ and for which all $1$-points lie on~$L$.

Suppose that $\F$ is not normal at $z_0\in D\cap L$ and let $(f_k)$ be a sequence
in $\F$ which does not have a subsequence converging in any neighborhood of~$z_0$.
Suppose that $(f_k|_{D^+})$ converges.
Then either $(f_k|_{D^+})\to 0$ and $(f_k|_{D^-})\to \infty$  or
$(f_k|_{D^+})\to \infty$ and $(f_k|_{D^-})\to 0$.

Let $z_0$ be as before and let $r>0$ with $\overline{D}(z_0,r)\subset D$.
Then for sufficiently large $k$ there exists a $1$-point $a_k$ of $f_k$ such that $a_k\to z_0$ and
if $M_k$ is the line orthogonal to $L$ which intersects $L$ at $a_k$, then
$|f_k(z)|\neq 1$ for all $z\in M_k\cap \overline{D}(z_0,r)\setminus \{a_k\}$.
\end{lemma}
The following lemma is a simple consequence of Harnack's inequality for the disk;
see~\cite[Harnack's Inequality,~3.6]{Axler1992}
\begin{lemma}\label{lemma1a}
Let $D\subset\C$ be a domain and let $K\subset D$ be compact.
Then there exists $C>1$ such that for any positive harmonic
function $u\colon D\to \R$ we have
\begin{equation}\label{mc28a}
\max_{z\in K} u(z) \leq C \min_{z\in K} u(z) .
\end{equation}
\end{lemma}
\begin{lemma}\label{lemma1}
Let $D$ be a domain and let $L$ be a straight line.
Let $\xi\in D\setminus L$ and let $K$ be a compact subset of~$D$.
Then there exists $C>0$ such that if $f\colon D\to\C$ is a holomorphic
function satisfying $|f(\xi)|>2$ which has no zeros in $D$ and
for which all $1$-points lie on~$L$
then $\log|f(z)|\leq C \log|f(\xi)|$ for all $z\in K$.
\end{lemma}
\begin{proof}
Without loss of generality we may assume that $L=\R$ and $\im \xi>0$.
Suppose that the conclusion is not true.
Then there exists a sequence $(f_k)$ of functions holomorphic in $D$ which satisfy
the hypotheses of the lemma and a sequence $(\zeta_k)$ in $K$ such that 
\begin{equation} \label{a2}
\frac{\log|f_k(\zeta_k)|}{\log |f_k(\xi)|}\to\infty.
\end{equation}
Since the $f_k$ have no zeros and $1$-points in $D\setminus\R$, the sequence $(f_k)$ is
normal in $D\setminus\R$.

Suppose first that the sequence $(f_k)$ is normal in~$D$.
If $|f_k(\xi)|\not\to\infty$, then there exists a subsequence of $(f_k)$ 
which tends to a limit function holomorphic in~$D$. This contradicts~\eqref{a2}.
Thus $|f_k(\xi)|\to\infty$ and hence $f_k|_D\to\infty$.
But then for large $k$ the functions $u_k$ given by 
\begin{equation} \label{a2a}
u_k(z):=\frac{\log|f_k(z)|}{\log|f_k(\xi)|}
\end{equation}
are positive harmonic functions
in some connected neighborhood of~$K$, and a contradiction to~\eqref{a2} is
now obtained from 
Lemma~\ref{lemma1a}.

We may thus assume that $(f_k)$ is not normal in~$D$. In fact, with
\begin{equation} \label{a3}
D^+:=\{z\in D\colon \im z>0\}
\quad\text{and}\quad
D^-:=\{z\in D\colon \im z<0\}
\end{equation}
we may assume $(f_k)$ converges in $D^+$ but that there exists 
$a\in D\cap \R$ such that no subsequence of $(f_k)$ converges in any neighborhood of~$a$.
It follows from Lemma~\ref{lemma2} that 
either $(f_k|_{D^+})\to 0$ and $(f_k|_{D^-})\to \infty$  or
$(f_k|_{D^+})\to \infty$ and $(f_k|_{D^-})\to 0$.
The first possibility is ruled out since we assumed that $\xi\in D^+$ and $|f_k(\xi)|>2$.
Thus
\begin{equation} \label{a3a}
(f_k|_{D^+})\to \infty
\quad\text{and}\quad
(f_k|_{D^-})\to 0.
\end{equation}
We may assume that $\zeta_k\to\zeta_0\in K$.
Lemma~\ref{lemma1a}
implies that the functions $u_k$ given by~\eqref{a2a}
are bounded on any compact subset of~$D^+$,
with a bound depending on this compact subset, but not on~$k$.
Together with~\eqref{a2} this yields that $\zeta_0\in \R$.
Without loss of generality we assume that $\zeta_0=0$.
Choose $\varepsilon>0$ such  that $\overline{D}(0,10\varepsilon)\subset D$.
We may assume that $\overline{D}(0,10\varepsilon)\subset K$.

Put 
\begin{equation} \label{a3b}
K^+_\varepsilon:=\{z\in K\colon \im z\geq \varepsilon\}
\quad\text{and}\quad
K^-_\varepsilon:=\{z\in K\colon \im z\leq -\varepsilon\}.
\end{equation}
Then by 
Lemma~\ref{lemma1a}
the functions $u_k$ are uniformly bounded on $K^+_\varepsilon$.
This means that there exists $M>1$ such that 
$\log |f_k(z)|\leq  M\log|f_k(\xi)|$ and hence
\begin{equation} \label{a6}
|f_k(z)|\leq  |f_k(\xi)|^M
\quad\text{for}\ z\in K^+_\varepsilon.
\end{equation}
By~\eqref{a3a} we also have
\begin{equation} \label{a6a}
|f_k(z)|> 1
\quad\text{for}\ z\in K^+_\varepsilon.
\end{equation}
and
\begin{equation} \label{a6b}
|f_k(z)|< 1<  |f_k(\xi)|^M
\quad\text{for}\ z\in K^-_\varepsilon,
\end{equation}
provided $k$ is large enough.

On the other hand, we have $|f_k(\zeta_k)|>  |f_k(\xi)|^M$  for large $k$
by~\eqref{a2}. Since $\zeta_k\to \zeta_0= 0$ we also have $|\zeta_k|<\varepsilon$
for large~$k$.
By the maximum principle, there exists a curve $\alpha_k$ connecting 
$\zeta_k$ with the circle $\{z\colon |z|=5\varepsilon\}$ such that 
\begin{equation} \label{a6c}
|f_k(z)|\geq |f_k(\zeta_k)|>|f_k(\xi)|^M>1
\quad\text{for}\ z\in\alpha_k.
\end{equation}
It follows from~\eqref{a6}, \eqref{a6b} and~\eqref{a6c} that 
\begin{equation} \label{a6d}
\alpha_k\subset 
\overline{D}(0,5\varepsilon)\setminus( K^+_\varepsilon\cap K^-_\varepsilon)
=\{z\colon |z|\leq 5\varepsilon,\; |\im z|<\varepsilon\}.
\end{equation}
It is no loss of generality to assume that $\alpha_k$ connects $\zeta_k$ with a 
point on the right arc of the boundary of the latter set; that is, 
$\alpha_k$ connects $\zeta_k$ with the
arc $\{z\colon |z|= 5\varepsilon,\; |\im z|<\varepsilon,\; \re z>0\}$;
see Figure~\ref{G_kH}.

\begin{figure}[!htb]
\captionsetup{width=.9\textwidth}
\centering
\begin{tikzpicture}[scale=4.1,>=latex](-1.1,-0.25)(1.1,1.1)
\draw[thin,dashed,->] (-1.1,0) -- (1.1,0);
\draw[thin,dashed,->] (0,-0.25) -- (0,1.05);
\draw[dotted] (1,0) arc (0.:180:1) ;
\draw (-0.5pt,-0.2) -- (0.5pt,-0.2);
\node[right] at (0.5pt,-0.2){$-\varepsilon$};
\foreach \x in {-5,1,2,...,5}
   {
    \draw (0.2*\x,-0.5pt) -- (0.2*\x,0.5pt);
   }
\foreach \x in {-5,2,3,4,5}
   {
    \node[below] at (0.2*\x,-0.5pt) {$\x \varepsilon$};
   }
    \node[below] at (0.2,-0.5pt) {$\varepsilon$};
\filldraw[black] (0.31,0) circle (0.012);
\node[above left] at (0.33,0-0.6pt) {$a_k$};
\filldraw[black] (0.68,0) circle (0.012);
\node[above right] at (0.66,0-0.6pt) {$b_k$};
\draw[-] (-0.98,0.2) -- (0.31,0.2);
\draw[-] (0.68,0.2) -- (0.98,0.2);
\draw[-] (0.68,0.2) -- (0.68,0.11);
\draw[-] (0.31,0.2) -- (0.31,0.12);
\filldraw[black] (0.31,0.12) circle (0.012);
\node[above right] at (0.31cm-0.6pt,0.11cm-0.8pt) {$u_k$};
\filldraw[black] (0.68,0.11) circle (0.012);
\node[above left] at (0.66cm+0.8pt,0.11cm-0.8pt) {$v_k$};
\node[right] at (0.31,1) {$M_k$};
\node[right] at (0.68,1) {$N_k$};
\draw[-,dotted] (0.31,-0.25) -- (0.31,1.05);
\draw[-,dotted] (0.68,-0.25) -- (0.68,1.05);
\draw (0.98,0.2) arc (11.54:168.46:1) ;
\draw [black, thick, densely dotted] plot [smooth, tension=0.5] coordinates { (0.31,0.12) (0.27,0.16) (0.18,0.17) (0.1,0.14) (0.0,0.07) (-0.13,-0.04) };
\draw [black, thick, densely dotted] plot [smooth, tension=0.5] coordinates {(0.68,0.11) (0.73,0.15) (0.82,0.14) (0.9,0.10) (1.00,0.03)};
\filldraw[black] (-0.13,-0.04) circle (0.012);
\node[below] at (-0.13,-0.04) {$\zeta_k$};
\node[above] at (0.08,0.00) {$\alpha_k$};
\draw [black, ultra thick] plot [smooth, tension=0.5] coordinates { (0.31,0.12) (0.4,-0.04) (0.5,-0.05) (0.6,0.02) (0.68,0.11)};
\node[below] at (0.5,-0.05){$\beta_k$};
\draw[thin,dashed,->] (-1.1,-1.4) -- (1.1,-1.4);
\draw[thin,dashed,->] (0,-1.65) -- (0,-0.35);
\draw (-0.5pt,-1.6) -- (0.5pt,-1.6);
\node[right] at (0.5pt,-1.6){$-\varepsilon$};
\foreach \x in {-5,1,2,...,5}
   {
    \draw (0.2*\x,-0.5pt-1.4cm) -- (0.2*\x,0.5pt-1.4cm);
   }
\foreach \x in {-5,4,5}
   {
    \node[below] at (0.2*\x,-1.4cm-0.5pt) {$\x \varepsilon$};
   }
\node[below] at (0.35,-1.4cm-0.5pt) {$2\varepsilon$};
\node[below] at (0.65,-1.4cm-0.5pt) {$3\varepsilon$};
\node[below] at (0.2,-1.4cm-0.5pt) {$\varepsilon$};
\draw[-] (-0.98,0.2-1.4) -- (0.4,0.2-1.4);
\draw[-] (0.6,0.2-1.4) -- (0.98,0.2-1.4);
\draw[-] (0.6,0.2-1.4) -- (0.6,-0.2-1.4);
\draw[-] (0.4,0.2-1.4) -- (0.4,-0.2-1.4);
\draw[ultra thick] (0.4,-0.2-1.4) -- (0.6,-0.2-1.4);
\draw (0.98,0.2-1.4) arc (11.54:168.46:1) ;
\node[below] at (0.5,-1.6){$\gamma$};
\end{tikzpicture}
\caption{The curves $\alpha_k$ and $\beta_k$ and the domains $G_k$ (top) and $H$ (bottom).}
\label{G_kH}
\end{figure}
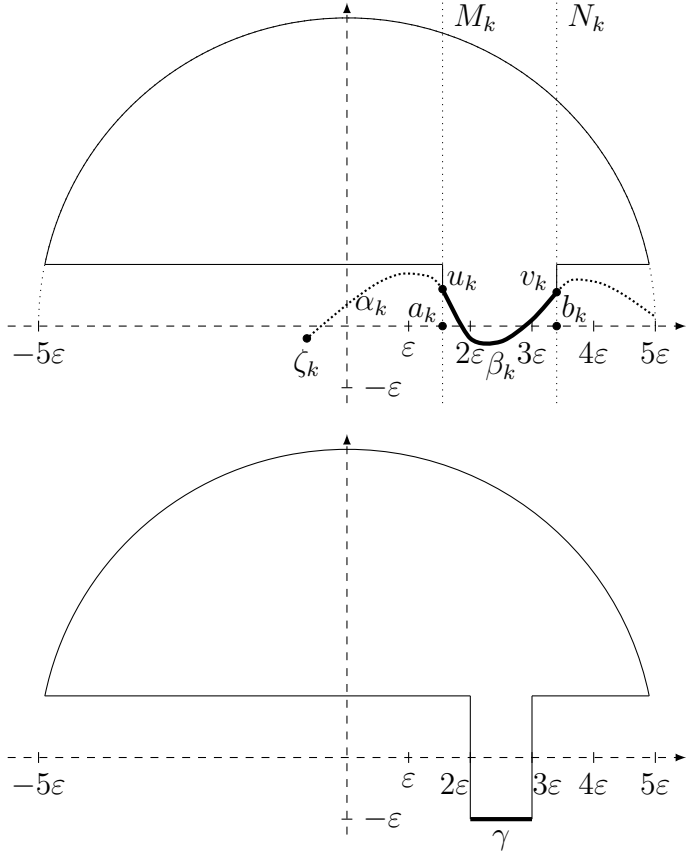

It follows from~\eqref{a3a} that no subsequence of $(f_k)$ is normal at any 
point of $D\cap \R$.
We may thus apply 
Lemma~\ref{lemma2} for any $z_0\in D\cap\R $. Since 
$D\cap\R\supset [-10\varepsilon,10\varepsilon]$ we may, in particular,
choose $z_0=3\varepsilon/2$ or $z_0=7\varepsilon/2$.
Doing so we find that if $k$ is sufficiently large, then 
there exist
$1$-points $a_k$ and $b_k$ of $f_k$ satisfying $\varepsilon<a_k<2\varepsilon$ 
and $3\varepsilon<b_k<4\varepsilon$ 
such that if $M_k$ and $N_k$ are the lines orthogonal to $\R$ which intersect
$\R$ in $a_k$ and $b_k$, respectively, then we have $|f_k(z)|\neq 1$ for 
$z\in M_k\cap\overline{D}(0,10\varepsilon)\setminus \{a_k\}$ and
$z\in N_k\cap\overline{D}(0,10\varepsilon)\setminus \{b_k\}$.
This implies that 
\begin{equation} \label{a6e}
|f_k(z)|>1
\quad\text{for}\ z\in (M_k\cup N_k)\cap \overline{D}(0,10\varepsilon)\cap D^+
\end{equation}
and $|f_k(z)|<1$ for $z\in (M_k\cup N_k)\cap \overline{D}(0,10\varepsilon)\cap D^-$.

The curve $\alpha_k$ intersects both lines $M_k$ and $N_k$. Let $\beta_k$ be the subcurve
of $\alpha_k$ which begins at the last intersection point of $\alpha_k$ with $M_k$ and 
ends at the first intersection point of $\alpha_k$ with $N_k$.
Note that by~\eqref{a6c}
the starting point $u_k$ of $\beta_k$ lies on $M_k\cap \{z\colon 0<\im z<\varepsilon\}$
while the end point $v_k$ lies on $N_k\cap \{z\colon 0<\im z<\varepsilon\}$.

We consider the domain $G_k$ bounded by the arc $\{z\colon |z|=5\varepsilon,\; \im z\geq\varepsilon\}$,
the horizontal line segments 
\begin{equation} \label{a6u}
\{x+i\varepsilon\colon -\sqrt{24}\varepsilon\leq x\leq a_k\}
\quad\text{and}\quad
\{x+i\varepsilon\colon  b_k\leq x\leq\sqrt{24}\varepsilon\},
\end{equation}
the vertical line segments 
\begin{equation} \label{a6v}
\{a_k+iy\colon \im u_k\leq y\leq\varepsilon\}
\quad\text{and}\quad
\{b_k+iy\colon \im v_k\leq y\leq\varepsilon\},
\end{equation}
and the curve $\beta_k$;
see Figure~\ref{G_kH}.

We want to show that the harmonic measure of $\beta_k$ in the domain 
$G_k$ at the point $2i\varepsilon$ is bounded away from~$0$.
This will be done by comparing it 
with the harmonic measure of the line segment
$\gamma:=\{x-i\varepsilon\colon 2\varepsilon\leq x\leq 3\varepsilon\}$ 
in the domain $H$ which is bounded by 
the horizontal line segments~$\gamma$,
\begin{equation} \label{a6x}
\{x+i\varepsilon\colon -\sqrt{24}\varepsilon\leq x\leq 2\varepsilon\}
\quad\text{and}\quad
\{x+i\varepsilon\colon  3\varepsilon\leq x\leq\sqrt{24}\varepsilon\},
\end{equation}
the vertical line segments 
\begin{equation} \label{a6y}
\{2\varepsilon+iy\colon -\varepsilon\leq y\leq\varepsilon\}
\quad\text{and}\quad
\{3\varepsilon+iy\colon -\varepsilon\leq y\leq\varepsilon\},
\end{equation}
and the arc $\{z\colon |z|=5\varepsilon,\; \im z\geq\varepsilon\}$;
see Figure~\ref{G_kH}.

For a bounded domain $G$ and a compact subset $A$ of $\partial G$, 
let $\omega(z,A,G)$ be the harmonic measure of $A$ at a point $z\in G$.
We want to show that 
\begin{equation} \label{a7}
\omega(2i\varepsilon, \gamma,H)\leq \omega(2i\varepsilon, \beta_k,G_k).
\end{equation}
In order to do so we note that (see~\cite[Corollary~ 4.3.9]{Ransford1995})
\begin{equation} \label{a7a}
\omega(z, \beta_k,G_k)
\geq \omega(z, \beta_k\cap \overline{H},G_k)
\geq \omega(z, \beta_k\cap \overline{H},G_k\cap H)
\end{equation}
for all $z\in G_k\cap H$. 

Another standard harmonic measure estimate (see~\cite[Lemma~ 2.6]{BE2020}) yields  that
\begin{equation} \label{a7b}
\omega(z, \beta_k\cap \overline{H},G_k\cap H)
\geq \omega(z, \gamma,H)
\end{equation}
for all $z\in G_k\cap H$. 
Combining this with~\eqref{a7a} we obtain~\eqref{a7}.

By~\eqref{a6a}, \eqref{a6c} and~\eqref{a6e} we have 
\begin{equation} \label{a7c}
|f_k(z)|> 1
\quad\text{for}\ z\in \partial G_k
\end{equation}
and large~$k$.
Since $\beta_k\subset \alpha_k$ we  have 
\begin{equation} \label{a6c1}
|f_k(z)|\geq |f_k(\zeta_k)|
\quad\text{for}\ z\in\beta_k
\end{equation}
by~\eqref{a6c}.
Using the Two-Constant Theorem \cite[Theorem~4.3.7]{Ransford1995}  we deduce
from~\eqref{a7c} and~\eqref{a6c1}  that 
\begin{equation} \label{a8}
\log|f_k(z)|\geq \omega(z,\beta_k,G_k) \log|f_k(\zeta_k)|
\end{equation}
for all $z\in G_k$.

Combining~\eqref{a8} with~\eqref{a6} and~\eqref{a7} we find that
\begin{equation} \label{a9}
\begin{aligned}
M\log|f_k(\xi)|
&\geq \log|f_k(2\varepsilon i)|
\\ &
\geq \omega(2\varepsilon i,\beta_k,G_k) \log|f_k(\zeta_k)|
\\ &
\geq \omega(2\varepsilon i,\gamma,H)  \log|f_k(\zeta_k)|.
\end{aligned}
\end{equation}
This contradicts~\eqref{a2}.
\end{proof}
\begin{proof}[Proof of Theorem~\ref{theorem1}]
Suppose first the the rays $L_0$, $L_1$ and $L_\infty$ are equally spaced.
Let $f$ be the function of Example~\ref{example1}. Then there exists $\theta\in\R$
such that either $g(z):=f(e^{i\theta}z)$ or $g(z):=1/f(e^{i\theta}z)$ defines a 
meromorphic function $g$ for which all zeros are on $L_0$, all $1$-points are 
on $L_1$ and all poles are on~$L_\infty$.
Thus for each $k\in\N$ the function $g_k$ defined by $g_k(z)=g(kz)$ is contained in $\F$.
It is easy to see that $\{g_k\colon k\in\N\}$ is not normal at
any point on any of the rays $L_0$, $L_1$ and $L_\infty$.
Thus $\F$ is not normal there.

Suppose now that $\F$ is not normal.
The rays $L_0$, $L_1$ and $L_\infty$ divide $A:=\{z\colon r<|z|<R\}$ into
three sectors which we denote by $S_0$, $S_1$ and~$S_\infty$.
Here $S_0$ is the sector ``opposite'' to $L_0$; that is, the sector bounded by 
$L_1$ and $L_\infty$. Similarly, $S_1$ and $S_\infty$ are the sectors opposite 
to $L_1$ and $L_\infty$, respectively.

By Montel's theorem, $\F$ is normal in 
$S_0\cup S_1\cup S_\infty =A\setminus(L_0\cup L_1\cup L_\infty)$.
Thus our assumption that $\F$ is not normal 
implies that $\F$ is not normal at some point in $A\cap(L_0\cup L_1\cup L_\infty)$.
Without loss of generality we may assume that $\F$ is not normal at some point
$z_1\in A\cap L_1$.
Let $(f_k)$ be a sequence in $\F$ which does not have a subsequence converging in any
neighborhood of~$z_1$. Since $\F$ is normal in $S_0$ and $S_\infty$, we may assume
that $(f_k)$ converges in $S_0$ and $S_\infty$. Lemma~\ref{lemma2} implies that 
either $(f_k|_{S_0})\to 0$ and $(f_k|_{S_\infty})\to \infty$  or
$(f_k|_{S_0})\to \infty$ and $(f_k|_{S_\infty})\to 0$.

This implies that $(f_k)$ is not normal at some point of $A\cap (L_0\cup L_\infty)$.
Assuming without loss of generality that $(f_k)$ is not normal at some point
$z_0\in A\cap L_0$, we deduce from Lemma~\ref{lemma2},
applied to $1-f_k$ instead of $f_k$,  that
either $(f_k|_{S_1})\to 1$ and $(f_k|_{S_\infty})\to \infty$  or
$(f_k|_{S_1})\to \infty$ and $(f_k|_{S_\infty})\to 1$.
The latter possibility contradicts our previous finding that either 
$(f_k|_{S_\infty})\to \infty$ or $(f_k|_{S_\infty})\to 0$.
Altogether we thus have 
$(f_k|_{S_0})\to 0$, $(f_k|_{S_1})\to 1$ and $(f_k|_{S_\infty})\to \infty$; that is,
\begin{equation} \label{a9a}
f_k(z)\to 
\begin{cases}
0 & \text{for}\ z\in S_0,\\
1 & \text{for}\ z\in S_1,\\
\infty & \text{for}\ z\in S_\infty.
\end{cases}
\end{equation}

Let now $\xi\in S_\infty$. Then $|f_k(\xi)|\to\infty$ as $k\to\infty$.  
Hence we may assume that $|f_k(\xi)|>2$ for all~$k\in\N$.
Lemma~\ref{lemma1} yields that the functions $u_k$ 
defined by
\begin{equation} \label{a10}
u_k(z):=\frac{\log|f_k(z)|}{\log|f_k(\xi)|}
\end{equation}
are locally bounded in $T_1:=S_0\cup S_\infty\cup (L_1\cap A)$.
Note that the $u_k$ are also harmonic in~$T_1$.
Passing to a subsequence if necessary we may thus assume that 
there exists a function $u$ harmonic in $T_1$ such that
\begin{equation} \label{a11}
u_k(z)\to u(z)
\quad\text{for}\ z\in T_1.
\end{equation}
Similarly, put 
\begin{equation} \label{a12}
v_k(z):=\frac{\log|f_k(z)-1|}{\log|f_k(\xi)|}.
\end{equation}
Then the $v_k$ are harmonic in $T_0:=S_1\cup S_\infty\cup (L_0\cap A)$.
Since $|f_k(\xi)|\to\infty$ we have $|f_k(\xi)-1|\to\infty$ and
$\log|f_k(\xi)|\sim \log|f_k(\xi)-1|$ as $k\to\infty$.
Using Lemma~\ref{lemma1} again we see that the functions
\begin{equation} \label{a12a}
z\mapsto \frac{\log|f_k(z)-1|}{\log|f_k(\xi)-1|}
\end{equation}
and hence the $v_k$ are locally bounded in $T_0$.
Passing to a subsequence if necessary we thus find 
that there exists a function $v$ harmonic in $T_0$ such that
\begin{equation} \label{a13}
v_k(z)\to v(z)
\quad\text{for}\ z\in T_0 .
\end{equation}
Moreover,
\begin{equation} \label{a14}
u(z)=v(z)
\quad\text{for}\ z\in T_0\cap T_1=S_\infty.
\end{equation}
We now consider the functions $w_k$ defined by 
\begin{equation} \label{a15}
w_k(z):=u_k(z)-v_k(z)=
\frac{1}{\log|f_k(\xi)|} \cdot \log\left|\frac{f_k(z)}{f_k(z)-1}\right|.
\end{equation}
These functions $w_k$ are harmonic in 
$T_\infty:=S_0\cup S_1\cup (L_\infty\cap A)$.
We have $w_k\to u$ in $S_0$ and $w_k\to -v$ in $S_1$. 
It follows that there is a function $w$ harmonic in $T_\infty$ such that 
\begin{equation} \label{a16}
w_k(z)\to w(z)
\quad\text{for}\ z\in T_\infty
\end{equation}
and
\begin{equation} \label{a17}
w(z)=
\begin{cases}
u(z) & \text{for}\ z\in S_0,\\
-v(z) & \text{for}\ z\in S_1.
\end{cases}
\end{equation}
Let now $S_\infty'$ and $S_\infty''$ be the two preimages of $S_\infty$ under $z\mapsto z^2$.
Analogously we define $S_0'$, $S_0''$, $S_1'$ and $S_1''$.
We may assume that these sectors are arranged in the cyclic order
$S_\infty',S_0',S_1',S_\infty'', S_0'',S_1''$.

We now define a function $h\colon A\to\R$ as follows:
\begin{equation} \label{a18}
h(z)=
\begin{cases}
v(z^2)=u(z^2) & \text{for}\ z\in S_\infty',\\
u(z^2)=w(z^2) & \text{for}\ z\in S_0', \\
w(z^2)=-v(z^2) & \text{for}\ z\in S_1', \\
-v(z^2)=-u(z^2) & \text{for}\ z\in S_\infty'', \\
-u(z^2)=-w(z^2) & \text{for}\ z\in S_0'', \\
-w(z^2)=v(z^2) & \text{for}\ z\in S_1''.
\end{cases}
\end{equation}
Here the two expressions used in the definition are equal by~\eqref{a14} and~\eqref{a17}.

It follows from~\eqref{a9a} that $u(z)\geq 0$ for $z\in S_\infty$ while
$u(z)\leq 0$ for $z\in S_0$. 
Since $u(\xi)=1$ we see that $u$ is non-constant and thus
$u(z)> 0$ for $z\in S_\infty$ while $u(z)< 0$ for $z\in S_0$. 
Analogously we see that 
$v(z)> 0$ for $z\in S_\infty$ while $v(z)< 0$ for $z\in S_1$. 
This implies that 
\begin{equation} \label{a19}
h(z)
\begin{cases}
>0 & \text{for}\ z\in S_\infty'\cup S_1'\cup S_0'',\\
<0 & \text{for}\ z\in S_0'\cup S_\infty''\cup S_1''.
\end{cases}
\end{equation}
Let $L$ be any ray separating two of the sectors
$S_\infty'$, $S_0'$, $S_1'$, $S_\infty''$, $ S_0''$ and $S_1''$. 
Thus $L$ is one of the preimages of $L_0$, $L_1$ or $L_\infty$
under $z\mapsto z^2$. Let $\sigma_L$ be the reflection in~$L$.
The reflection principle for harmonic functions~\cite[Theorem~ 1.2.9]{Ransford1995}
implies that $h(\sigma_L(z))=-h(z)$. This implies that all sectors 
$S_\infty'$, $S_0'$, $S_1'$, $S_\infty''$, $ S_0''$ and $S_1''$ have 
the same opening angle. 
It follows that $S_0$, $S_1$ or $S_\infty$ all have opening angle $2\pi/3$.
Thus the rays $L_0$, $L_1$ or $L_\infty$ are equally spaced.
\end{proof}
\section{Proof of Theorem~\ref{theorem2}} \label{sec-proof2}
As mentioned in the introduction,
normal functions cannot be dealt with by Theorem~\ref{theorem1}.
The results of Ostrowski~\cite{Ostrowski1925} already mentioned imply 
in particular that normal functions have order~$0$.
The following result actually covers functions of order less than~$1$.
\begin{proposition}\label{proposition1}
Let $L_0$, $L_1$ and $L_\infty$ be three distinct rays emanating from the origin. 
Then there is no transcendental meromorphic function of order less than~$1$
for which all but finitely many zeros lie on $L_0$, all but finitely many $1$-points lie on $L_1$ and all
but finitely many poles lie on $L_\infty$.
\end{proposition}
To prove this proposition, we will use the following lemma.
This lemma may be known, but since we are not aware of any reference, we include 
a detailed proof.
\begin{lemma}\label{lemma3}
Let $a,b,p,q\in\C\setminus\{0\}$ and suppose that
$1$, $p$ and $q$ are distinct. Then there exists
$\delta\in(0,\pi)$ such that for some arbitrarily large $n\in\N$ the points
$1$, $ap^n$ and $bq^n$ lie in a sector opening angle~$\delta$.
\end{lemma}
Trivially, there is a half-plane (that is, a sector of opening $\pi/2$) containing~$1$,
$ap^n$ and $bq^n$. The point is that $\delta$ is strictly less than  $\pi/2$.

To prove Lemma~\ref{lemma3}, we will use several other lemmas.
\begin{lemma}\label{lemma4}
Let $A,B\in\partial \D$
 such that $\re (A+B)>0$. 
Then $1$, $A$ and $B$ lie on an arc of $\partial\D$ of length at most 
$\arccos(\re (A+B)-1)$.
\end{lemma}
\begin{proof}
The hypothesis implies that $A\neq -1$ and $B\neq -1$.
We may assume that $\im B\geq 0$, since otherwise we can pass to the 
complex conjugates of $A$ and~$B$. We may thus 
write $A=e^{i\alpha}$ and $B=e^{i\beta}$ with $\alpha\in (-\pi,\pi)$
and $\beta\in [0,\pi)$. Then
\begin{equation} \label{c5a}
\re (A+B)=\cos\alpha+\cos\beta=2\cos\frac{\alpha+\beta}{2} \cos\frac{\beta-\alpha}{2}.
\end{equation}
Suppose first that $\alpha<0$. Then $-\pi<\alpha+\beta<\pi$ and thus
$\cos((\alpha+\beta)/2)>0$. Hence $\cos((\beta-\alpha)/2)>0$ so that
$0\leq|\alpha+\beta|\leq\beta-\alpha<\pi$.
Hence
\begin{equation} \label{c5b}
\re (A+B)\geq 2\cos^2\frac{\beta-\alpha}{2}=1+\cos(\beta-\alpha).
\end{equation}
As the arc on $\partial \D$
which connects $A$ with $B$ and contains $1$ has length $\beta-\alpha$, the
conclusion follows.

Suppose now that $\alpha\geq 0$. We may suppose that $\alpha\leq \beta$.
Then there is an arc on $\partial \D$ of 
length $\beta$ which contains $1$, $A$ and~$B$.
Now~\eqref{c5a} yields that
\begin{equation} \label{c5c}
\re (A+B)\geq 2\cos^2\frac{\alpha+\beta}{2}=1+\cos(\alpha+\beta).
\end{equation}
Thus $\beta\leq\alpha+\beta\leq \arccos(\re (A+B)-1)$, and 
again the conclusion follows.
\end{proof}
For a sequence $(z_n)_{n\in\N}$ in $\C$, we define the \emph{average} 
\begin{equation} \label{c6}
\A((z_n)_{n\in\N}) :=\lim_{n\to\infty}\frac{1}{n}\sum_{k=1}^nz_k,
\end{equation}
provided that the limit exists.
For $c\in\C$ and $\xi\in\partial \D\setminus\{1\}$ we have 
\begin{equation} \label{c7}
\A((c\xi^n)_{n\in\N}) =\lim_{n\to\infty}\frac{1}{n}c\sum_{k=1}^n\xi^n
=\lim_{n\to\infty}\frac{1}{n}c\,\frac{\xi-\xi^{n+1}}{1-\xi}=0.
\end{equation}
Taking the real part yields for $c=e^{i\gamma}$ and $\xi=e^{i\tau}$ that
\begin{equation} \label{c8}
\A((\cos(\gamma+n\tau))_{n\in\N}) =
\begin{cases}
0 & \text{if}\ \tau \not\equiv 0\pmod {2\pi},\\
\cos\gamma & \text{if}\ \tau \equiv 0\pmod {2\pi}.\\
\end{cases}
\end{equation}
We will use the following lemma.
\begin{lemma}\label{lemma5}
Let $(x_n)_{n\in\N}$ be a bounded real sequence
satisfying $\A((x_n))=0$.
If $\A((x_n^2))$ exists, then
\begin{equation} \label{c8b}
\limsup_{n\to\infty} x_n \cdot \limsup_{n\to\infty} |x_n| \geq \A((x_n^2)).
\end{equation}
\end{lemma}
\begin{proof}
Let $\alpha,\beta>0$ and suppose that $x_n\leq \alpha$ and 
$|x_n|\leq \beta$ for all large $n\in\N$, say for $n\geq N$.
Then $(x_n-\alpha)^2\leq -(\alpha+\beta)(x_n-\alpha)$ for all $n\geq N$ and thus
\begin{equation} \label{c8a}
\begin{aligned}
&\ \quad 
\frac{1}{n}\sum_{k=N}^n x_k^2
-2\alpha \frac{1}{n}\sum_{k=N}^n x_k +\alpha^2 \frac{n-N+1}{n}
% \\ &
=\frac{1}{n}\sum_{k=N}^n (x_k-\alpha)^2
\\ &
\leq -(\alpha+\beta) \frac{1}{n}\sum_{k=N}^n (x_k-\alpha)
%\\ &
= (\alpha+\beta)\alpha\frac{n-N+1}{n}
-(\alpha+\beta) \frac{1}{n}\sum_{k=N}^n x_k .
\end{aligned}
\end{equation}
It follows that $\A((x_n^2))+\alpha^2\leq (\alpha+\beta)\alpha$ and hence that
\begin{equation} \label{c9}
\A((x_n^2)) \leq  \alpha\beta,
\end{equation}
from which the conclusion follows.
\end{proof}
\begin{proof}[Proof of Lemma~\ref{lemma3}]
Since the conclusion depends only on the arguments of $a$, $b$, $p$ and $q$, 
we may assume that $|a|=|b|=|p|=|q|=1$.
We write $a=e^{i\alpha}$, $b=e^{i\beta}$, $p=e^{i\phi}$ and $q=e^{i\psi}$,
with $\alpha,\beta,\phi,\psi\in(-\pi,\pi]$.
Since $1$, $p$ and $q$ are distinct we have
\begin{equation} \label{c9a}
\phi\neq 0,\quad \psi\neq 0\quad\text{and}\quad \phi\neq\psi.
\end{equation}
We will apply Lemma~\ref{lemma5} with
\begin{equation} \label{c10}
x_n:=\re(ap^n+bq^n) =\cos(\alpha+n\phi) +\cos(\beta+n\psi).
\end{equation}
It follows from~\eqref{c8} that $\A((x_n))=0$. We have
\begin{equation} \label{c11}
\begin{aligned} 
x_n^2
&=\cos^2(\alpha+n\phi)+\cos^2(\beta+n\psi)+2\cos(\alpha+n\phi)\cos(\beta+n\psi)
\\ &
=1+\frac12\cos(2\alpha+2n\phi)+\frac12\cos(2\beta+2n\psi)+\cos(\alpha+\beta+n(\phi+\psi))
\\ &
\qquad
+\cos(\alpha-\beta+n(\phi-\psi))
\end{aligned} 
\end{equation}
Suppose first that $2\phi\not\equiv 0 \pmod {2\pi}$ and
$2\psi\not\equiv 0 \pmod {2\pi}$. Equivalently, $\phi\neq\pi$ and $\psi\neq\pi$.
It follows from~\eqref{c8} and~\eqref{c9a} that 
\begin{equation} \label{c12}
\A((x_n^2))=
\begin{cases}
1 &\text{if}\ \phi\neq -\psi,\\
1+\cos(\alpha+\beta) &\text{if}\ \phi= -\psi.
\end{cases}
\end{equation}
Thus $\A((x_n^2))>0$ unless $\phi= -\psi$ and $\alpha+\beta\equiv \pi\pmod {2\pi}$.
Postponing this exceptional case, 
and noting that $\limsup_{n\to\infty}|x_n|\leq 2$ by~\eqref{c10},
we deduce from Lemma~\ref{lemma5} that there exist arbitrarily large $n\in\N$ such
that $x_n\geq \A((x_n^2))/4$.
Lemma~\ref{lemma4} implies that 
for such $n$ the points $1$, $ap^n$ and $bq^n$ are contained in 
an arc of length $\arccos(\A((x_n^2))/4-1)$.
In this case we may thus take $\delta=\arccos(\A((x_n^2))/4-1)$.

Suppose now that $2\phi\equiv 0 \pmod {2\pi}$. Then $\phi=\pi$ and thus $\phi\neq -\psi$.
Hence
\begin{equation} \label{c13}
\A((x_n^2))=
1+\frac12\cos(2\alpha)>0,
\end{equation}
and the conclusion follows as before.
The case that $2\psi\equiv 0 \pmod {2\pi}$ and thus $\psi=\pi$ is analogous.

It remains to consider the case that  $\phi= -\psi$ and $\alpha+\beta\equiv \pi\pmod {2\pi}$.
Then $q=\overline{p}$ and $b=-\overline{a}$.
Thus $ap^n$ and $bq^n$ are symmetric with respect to the imaginary axis
so that $\im(ap^n)$ and $\im(bq^n)$ have the same sign.
If  $\delta$ with the properties claimed does not exist, we thus must have
\begin{equation} \label{c14}
\min\{|ap^n+1|,|bq^n+1|\}=
\min\{|ap^n+1|,|ap^n-1|\}\to 0.
\end{equation}
Thus the only accumulation points of the sequence $(ap^n)$ are $\pm 1$.
This implies that $p=\pm 1$ and $q=\mp 1$, contradicting the hypothesis 
that $1$, $p$ and $q$ are distinct.
\end{proof}
\begin{lemma}\label{lemma13}
Let $F$, $G$ and $H$ be transcendental entire functions for which the arguments of
the Taylor coefficients tend to~$0$.
Let $p,q,r\in\partial\D$ be distinct.
Then $F(pz)$, $G(qz)$ and $H(rz)$ are linearly independent.
\end{lemma}
\begin{proof}
Let
\begin{equation} \label{f3}
F(z)=\sum_{n=0}^\infty \alpha_n z^n,
\quad
G(z)=\sum_{n=0}^\infty \beta_n z^n
\quad\text{and}\quad
H(z)=\sum_{n=0}^\infty \gamma_n z^n.
\end{equation}
The hypothesis says that 
\begin{equation} \label{f4}
\arg \alpha_n\to 0, \quad \arg \beta_n\to 0
\quad\text{and}\quad \arg \gamma_n\to 0 
\end{equation}
as $n\to\infty$.
Suppose now that 
\begin{equation} \label{f5}
aF(pz)+bG(qz)+cH(rz)=0,
\end{equation}
with $a,b,c\in \C$.
If $c=0$, then we easily obtain $a=b=0$. 
Thus suppose that $c\neq 0$. Without loss of generality we may assume that $c=1$.
We may also assume that $r=1$.
It follows that 
\begin{equation} \label{f6}
\alpha_n a p^n + \beta_n b q^n + \gamma_n =0
\end{equation}
for all $n\geq 0$. 
Lemma~\ref{lemma3} implies that there exists arbitrarily large $n$ such that 
the arguments of $a p^n$, $b q^n$ and $1$ lie in an interval of length at most~$\delta$.
It thus follows from~\eqref{f4} that the 
 arguments of $\alpha_n a p^n$, $\beta_n b q^n$ and $\gamma_n$ lie in an interval of length 
less than~$\pi$.
This contradicts~\eqref{f6}.
\end{proof}
\begin{lemma}\label{lemma15}
Let $F$ be an entire function of the form
\begin{equation} \label{g1}
F(z)=P(z)\prod_{k=1}^\infty \left(1+\frac{z}{x_k}\right),
\end{equation}
where $(x_k)$ is a sequence of positive numbers tending to $\infty$ and 
where $P$ is a polynomial with positive leading coefficient.
Then the arguments of the Taylor coefficients of $F$ tend to~$0$.
\end{lemma}
\begin{proof}
Let
\begin{equation} \label{g2}
\prod_{k=1}^\infty \left(1+\frac{z}{x_k}\right)=\sum_{n=0}^\infty a_n z^n ,
\quad
P(z)=\sum_{n=0}^d b_n z^n
\quad\text{and}\quad
F(z)=\sum_{n=0}^\infty c_n z^n.
\end{equation}
Then 
\begin{equation}\label{g3}
c_n=\sum_{k=0}^d b_ka_{n-k}.
\end{equation}
It is well-known and easy to prove that $a_n>0$ and  $a_n^2>a_{n-1}a_{n+1}$
for all $n\in\N$.
(More generally, the sequence $(a_n)$ is totally positive; see~\cite{ASW}.)
Thus  $a_{n+1}/a_n$ is decreasing.
Since $F$ is entire, this implies that $a_{n+1}/a_n\to 0$.

Dividing \eqref{g3} by $a_{n-d}$ we find that 
\begin{equation}\label{g4}
\frac{c_n}{a_{n-d}}=b_d+\sum_{k=0}^{d-1}b_k\frac{a_{n-k}}{a_{n-d}}\rightarrow b_d.
\end{equation}
as $n\to\infty$.
Since $a_{n-d}>0$ and $b_d>0$ we conclude that $\arg c_n\to 0$.
\end{proof}
\begin{proof}[Proof of Proposition~\ref{proposition1}]
Let $f$ be a transcendental meromorphic function of order less than~$1$
for which all but finitely many zeros lie on $L_0$, all but finitely many $1$-points 
lie on $L_1$ and all but finitely many poles lie on $L_\infty$.
Without loss of generality we may assume 
that $f(0)\in\C\setminus\{0\}$.
Let $\Pi_0$, $\Pi_1$ and $\Pi_\infty$ be the canonical products of the zeros,
$1$-points and poles. Then 
\begin{equation} \label{c1}
f=f(0)\frac{\Pi_0}{\Pi_\infty}
\quad\text{and}\quad
f-1=C\frac{\Pi_1}{\Pi_\infty}
\end{equation}
for some constant~$C$. It follows that 
\begin{equation} \label{c2}
f(0)\Pi_0=C\Pi_1+\Pi_\infty.
\end{equation}
Let $-\overline{p}$, $-\overline{q}$
and $-\overline{r}$ be the points where the rays $L_0$, $L_1$ and $L_\infty$
intersect~$\partial \D$.
Then $\Pi_0$ can be written in the form
$\Pi_0(z)=a F(pz)$ where $F$ satisfies the hypothesis of Lemma~\ref{lemma15} and
$a\in\C\setminus\{0\}$.
Similarly, 
$\Pi_1(z)=b G(pz)$
and
$\Pi_\infty(z)=c H(pz)$
for entire functions $G$ and $H$ satisfying the hypothesis of Lemma~\ref{lemma15},
and $b,c\in\C\setminus\{0\}$.
Equation~\eqref{c2} says that $F(pz)$, $G(qz)$ and $H(rz)$ are linearly dependent.
This contradicts Lemma~\ref{lemma13}.
\end{proof}
\begin{proof}[Proof of Theorem~\ref{theorem2}]
Let $f$ be a transcendental meromorphic function 
for which all but finitely many zeros lie on $L_0$, all but finitely many $1$-points lie on $L_1$ and all
but finitely many poles lie on $L_\infty$.
Proposition~\ref{proposition1} implies that $f$ has order at least~$1$.
The results of Ostrowski~\cite{Ostrowski1925} already quoted yield that 
the family $\{f(rz)\colon r>0\}$ is not normal in $\C\setminus\{0\}$.
The conclusion now follows from Theorem~\ref{theorem1}.
\end{proof}
\section{Proof of Theorem~\ref{theorem3}}\label{sec-proof3}
Let $g\colon [r_0,\infty)\to\R$ be a positive increasing function
and $\lambda\geq 0$.
A sequence $(r_k)$ tending to $\infty$ is called a sequence of
P\'olya peaks (of the first kind) of order $\lambda$ for $g$ if
for every $\varepsilon>0$, we have 
\begin{equation}\label{5a}
 g(tr_k)\leq (1+\varepsilon)t^{\lambda} g(r_k)
\quad\text{for}\  \varepsilon\leq t\leq \frac{1}{\varepsilon}
\end{equation}
for all large~$k$.
If instead of~\eqref{5a} we have 
\begin{equation}\label{5b}
 g(tr_k)\geq (1-\varepsilon)t^{\lambda} g(r_k)
\quad\text{for}\  \varepsilon\leq t\leq \frac{1}{\varepsilon}
\end{equation}
for all large~$k$, then $(r_k)$ is called a sequence of
P\'olya peaks of the second kind (of order $\lambda$ for~$g$).

Put
\begin{equation}\label{rho1}
\rho^*:=\sup\left\{ p\in {\mathbb R}  \colon  \limsup_{r,t\to\infty}
\frac{ g(tr)}{t^p g(r)}=\infty\right\}
\end{equation}
and
\begin{equation}\label{rho2}
\rho_*:=\inf\left\{ p \in {\mathbb R}  \colon  \liminf_{r,t\to\infty}
\frac{ g(tr)}{t^p g(r)}=0\right\} .
\end{equation}
Then
\begin{equation}\label{rr}
0\leq\rho_*
\leq\liminf_{r\to\infty}\frac{\log g(r)}{\log r}
\leq\limsup_{r\to\infty}\frac{\log g(r)}{\log r}
\leq\rho^*\leq\infty .
\end{equation}
The upper and lower limits in~\eqref{rr} are called the order and lower order of~$g$.
For a meromorphic function
$f$ the order and lower order are obtained by taking for $g(r)$ the 
Nevanlinna characteristic $T(r,f)$.

The following result is due to  Drasin and Shea~\cite{Drasin1972}.
\begin{lemma}\label{lemma6}
Let $g\colon [r_0,\infty)\to\R$ be a positive increasing function and $\lambda\geq 0$.
Then the following are equivalent:
\begin{itemize}
\item[$(a)$] $\rho_*\leq\lambda\leq\rho^*$.
\item[$(b)$] $g$ has P\'olya peaks of the first kind of order $\lambda$.
\item[$(c)$] $g$ has P\'olya peaks of the second kind of order $\lambda$.
\end{itemize}
\end{lemma}
We will also use the following standard result about positive harmonic functions.
\begin{lemma}\label{lemma12}
Let $u$ be a positive harmonic function in the right half-plane which extends
continuously to $i\R\setminus\{0\}$, with $u(iy)=0$ for $y\in \R\setminus\{0\}$.
Then $u(z)=\re(az+b/z)$ with $a,b\geq 0$.
\end{lemma}
To prove this result, we note that~\cite[Theorem~7.26]{Axler1992} yields that 
$u$ has the form $a \re z+P(z)$ where $P$ is a Poisson integral for the right
half-plane. Applying~\cite[Theorem~7.19]{Axler1992} to $u(z)-a\re z$ shows that
$P$ has the form $P(z)=\re(b/z)$.
  
Lemmas~\ref{lemma6} and~\ref{lemma12}
 will be used to prove that the Schwarzian $S(f)$ is rational.
In order to prove that $S(f)$ is not only rational, but has the form~\eqref{i3},
we need results of 
Elfving~\cite{Elfving1934} 
and Nevanlinna~\cite{Nevanlinna1932} concerning meromorphic functions 
with rational Schwarzian derivative. These results were proved by
Nevanlinna for 
the case of a polynomial Schwarzian derivative and extended to 
rational Schwarzian derivatives by Elfving.

The first result we need is the following.
\begin{lemma}\label{lemma9}
Let $Q$ be a rational function satisfying  $Q(z)\sim a z^d$
as $z\to\infty$, with $d\in\N$ and $a\in\C\setminus\{0\}$.
Let $f$ be a meromorphic function satisfying
\begin{equation} \label{d13}
S(f)=Q.
\end{equation}
Then $f$ has order $(d+2)/2$.
\end{lemma}

We will see that in our case the order of $f$ is $3/2$ so that $d=1$. Thus
$Q(z)\sim a z$ as $z\to\infty$.
It is no loss of generality to assume that $a<0$.
The asymptotics of $f$ are then described by the following result.
\begin{lemma}\label{lemma11}
For $j\in\{1,2,3\}$,
let $L^j=\{re^{i(2j-1)\pi/3}\colon r>0\}$. The rays $L^j$ divide $\C$ into
three congruent sectors. Let $V_j$ be the sector opposite to~$L^j$.

Let $c>0$ and let $Q$ be a rational function satisfying
\begin{equation} \label{d13a}
Q(z)\sim -c z
\quad\text{as}\  z\to\infty.
\end{equation}
Let $f$ be a meromorphic function satisfying~\eqref{d13}.

Then there exist distinct values $a_1,a_2,a_3\in\CC:=\C\cup\{\infty\}$ such that $f(z)\to a_j$ 
as $|z|\to\infty$ in any closed subsector of~$V_j$.
These values $a_j$ are logarithmic singularities, and $f$ has no other asymptotic values.

For each  $j\in\{1,2,3\}$, the function $f$ has infinitely any $a_j$-points,
and given $\varepsilon>0$, all but finitely many $a_j$-points 
are contained in the sector of opening angle $\varepsilon$ bisected by~$L^j$.

Moreover, a meromorphic function $F$ satisfies $S(F)=Q$ if and only if $F$
is of the form $F=L\circ f$ with a linear fractional transformation~$L$.
\end{lemma}
Replacing $f$ by $L\circ f$ with a linear fractional transformation~$L$
we can replace the values $a_1$, $a_2$ and $a_3$ 
by three other distinct values, in particular by the values $0$, $1$ and~$\infty$.

The following result is due to Gundersen~\cite[Theorem~3]{Gundersen1986}.
Here a meromorphic function $f$ is called real if $f(x)\in\R\cup\{\infty\}$ 
for all $x\in\R$. Otherwise it is called nonreal.
\begin{lemma}\label{lemma7}
Let $A$ be a nonreal polynomial of degree~$n$, put
\begin{equation} \label{d10}
F(z)=\frac{A(z)-\overline{A(\overline{z})}}{2i},
\end{equation}
and let $p$ denote the number of distinct real zeros of~$F$.
Let $w$ be a nontrivial solution of 
\begin{equation} \label{d11}
w''+Aw=0.
\end{equation}
Then the number $k$ of real zeros of $w$ is finite and we have $k\leq p+1$.
In particular, $k\leq n+1$.
\end{lemma}
If $A$ is a polynomial, then every solution $w$ of~\eqref{d11} is entire.
It follows from Lemma~\ref{lemma7} that if there is a solution of~\eqref{d11}
which has infinitely many real zeros, then $A$ is real.

If $w_1$ and $w_2$ are linearly independent solutions of~\eqref{d11},
then $f:=w_1/w_2$ satisfies 
\begin{equation} \label{d12}
S(f)=2A.
\end{equation}
Conversely, every solution $f$ of~\eqref{d12} is a quotient of two linearly
independent solutions of~\eqref{d11}. Thus we find that if $f$ satisfies~\eqref{d12}
for some polynomial $A$ and if $f$ has infinitely many real zeros,
then $A$ is real.
Since $S(L\circ f)=S(f)$ for every linear fractional transformation $L$ 
we see that if a meromorphic function $f$
satisfying~\eqref{d12} has infinitely many real $a$-points 
for some $a\in\CC$, then $A$ is real.

It turns out that this remains valid for rational functions~$A$.
\begin{lemma}\label{lemma8}
Let $Q$ be a rational function and let $f$ be a meromorphic functions
satisfying $S(f)=Q$. If $f$ has infinitely many real $a$-points for some
$a\in\CC$, then $Q$ is real.
\end{lemma}
As explained above, this result follows from Lemma~\ref{lemma7} if $Q$ is a 
polynomial. However, the proof extends to the case that $Q$ is rational.
We note that in order to prove Lemma~\ref{lemma8} for rational $Q$
it does not suffice to extend Lemma~\ref{lemma7} to the case
that $A$ is rational and $w$ is meromorphic, since for rational $A$ the solutions
of~\eqref{d11} may be multi-valued, but the quotient of two multi-valued 
solutions may be single-valued. However, the proof of Lemma~\ref{lemma7}
given in~\cite{Gundersen1986} also extends to multi-valued functions.

The proof of the following lemma uses Lommel's method 
to prove that the zeros of Bessel functions are real; see~\cite[p.~482]{Watson1966}.
\begin{lemma} \label{lemma14}
Let $r>0$,  $\gamma> -2$ and $0<\alpha< \pi$ with  $(2+\gamma)\alpha<\pi$.
Let $u$ and $A$ be holomorphic in
$S:=\{z\colon |z| > r ,\, |\arg z|< \alpha\}$
and suppose that
$u''+Au=0$.
Suppose also that both $u$ and $A$ are real on the real axis and that there exists $c>0$ 
such that
\begin{equation} \label{y1}
A(z)\sim cz^\gamma
\quad\text{as}\ z\to\infty,\, z\in S.
\end{equation}
Then there exists $x_1>r$ such that all zeros of $u$ in $\{z\colon |\arg(z-x_1)|<\alpha\}$
are real.
\end{lemma}
\begin{proof}
A classical result of Kneser~\cite[Section~6]{Kneser1893} 
implies that $u$ has arbitrarily large positive zeros.
(Kneser's result says that this is the case if there exists $\delta>0$
such that $x^2A(x)\geq 1/4+\delta$ for all large~$x$.)

It follows from~\eqref{y1} that $\arg A(z)=\gamma\arg z+o(1)$
as $z\to\infty$. 
Since $\arg A(x)=\gamma\arg x=0$ for $x>r$ this actually implies that
\begin{equation} \label{y3}
\arg A(z)=(\gamma+o(1))\arg z
\quad\text{as}\ z\to\infty,\, z\in S.
\end{equation}
If $x_1$ is large and $|\arg z|<\alpha$, then $|x_1+z|$ is also large.
Thus~\eqref{y3} yields that
\begin{equation} \label{y4}
\arg\!\left( z^2 A(z+x_1)\right)=2\arg z + (\gamma+o(1))\arg (z+x_1)
\end{equation}
as $x_1\to\infty$.
Since $|\arg (z+x_1)|<|\arg z|$, 
it now follows from the hypotheses $\gamma>-2$ and  $(2+\gamma)\alpha<\pi$ that 
\begin{equation} \label{y5}
\im\!\left( z^2 A(z+x_1)\right)>0 
\quad\text{for}\ z\in S\ \text{with}\ \im z>0,
\end{equation}
provided $x_1$ is sufficiently large.

Put $v(z):=u(x_1+z)$ and $B(z):=A(x_1+z)$,
with a large zero $x_1$ of~$u$.
Then $v''+Bv=0$ and $v(0)=0$.
Let $a,b\in S-x_1=\{z-x_1\colon z\in S\}$. Then
\begin{equation} \label{y6}
\begin{aligned}
\frac{d}{dt}\left( a\,v'(at)v(bt)-b\,v(at)v'(bt) \right)
&=a^2 v''(at)v(bt)- b^2 v(at)v''(bt)
\\ &
= - \left(a^2 B(at)-b^2 B(bt) \right) v(at)v(bt) .
\end{aligned}
\end{equation}
Let now $\xi\in S_1:= \{z\colon |\arg(z-x_1)|<\alpha\}$ be a non-real zero of~$u$.
Then $\overline{\xi}$ is also a zero of~$u$.  We may assume that $\im \xi>0$.
With $a:=\xi-x_1\in S$ and $b:=\overline{\xi}-x_1\in S$
we have $v(a)=v(b)=0$.
It follows from~\eqref{y6} that
\begin{equation} \label{y7}
\begin{aligned} 
0
&= \int_0^1\left(a^2 B(at)-b^2 B(bt) \right) v(at)v(bt) dt
\\ &
=2i\int_0^1 \im\!\left(a^2 B(at)\right) |v(at)|^2 dt.
\end{aligned} 
\end{equation}
By~\eqref{y5} we have 
\begin{equation} \label{y8}
\im\!\left(a^2 B(at)\right)=\frac{1}{t^2} \im\!\left((ta)^2 A(at+x_1)\right)>0.
\end{equation}
This contradicts~\eqref{y7}.
Thus all zeros in the sector $S_1$ lie on the positive axis.
\end{proof}

\begin{remark} \label{remark1}
Considering $u(-z)$ instead of $u(z)$ we see that
Lemma~\ref{lemma14} remains valid if we put
$S:=\{z\colon |z|>r ,\, |\arg z-\pi|\leq \alpha\}$
and assume that there exists $c>0$ such that 
$A(z)\sim -cz^\gamma$ as $z\to\infty$ in~$S$. 
\end{remark}

\begin{proof}[Proof of Theorem~\ref{theorem3}]
Suppose first that there exist $a_1$, $a_2$ and $a_3$ such that 
all but finitely many $a_j$-points are on the ray $L^j$.
We may assume that $a_1=0$, $a_2=1$ and $a_3=\infty$, since otherwise we can replace
$f$ by $L\circ f$ with a suitable linear fractional transformation~$L$.
We switch to the notation previously used by putting
$L_0=L^1$, $L_1=L^2$ and $L_\infty=L^3$.

Let
\begin{equation} \label{defn}
n(r):=n(r,0)+n(r,1)+n(r,\infty).
\end{equation}
and let $\rho_*$ and $\rho^*$ be defined by~\eqref{rho1} and~\eqref{rho2},
with $g(r)$ replaced by~$n(r)$.

Since $f$ has at most two Borel exceptional
values~\cite[Chapter~3, Theorem~2.2]{Goldberg2008},
the order of $n(r)$ is equal to that of~$f$.
Since $f$ has order at least~$1$ by Proposition~\ref{proposition1},
Lemma~\ref{lemma6} yields that $\rho^*\geq 1$.

For a sequence $(r_k)$ tending to $\infty$, we consider the sequence $(f_k)$ 
defined by $f_k(z)=f(r_k z)$.
We will prove the following: 
\begin{itemize}
\item[$(a)$]
If $(f_k)$ is normal in  $\C\setminus\{0\}$,
then $(r_k)$ has a subsequence which is 
a sequence of P\'olya peaks of order $0$ for $n(r)$.
\item[$(b)$]
$\rho_*\leq 3/2$.
\item[$(c)$]
If $(r_k)$ is a sequence of P\'olya peaks for $n(r)$ of finite non-zero order $\lambda$,
then $\lambda=3/2$.
\end{itemize}
Since $\rho^*\geq 1$ we can deduce from $(b)$, $(c)$ and
Lemma~\ref{lemma6} that $\rho_*=\rho^*=3/2$.
This implies that $n(r)$ and hence $f$ have order~$3/2$.
Moreover, it follows from $(a)$ that if $(r_k)$ is a sequence tending to~$\infty$,
then the sequence $(f_k)$ cannot be normal in $\C\setminus\{0\}$.

To prove $(a)$, let $(f_k)$ be normal in  $\C\setminus\{0\}$.
Passing to a subsequence if necessary, we may assume that $(f_k)$ converges, say
$f_k(z)\to \phi(z)$ in $\C\setminus\{0\}$.
Let $0<\varepsilon<1$ and let $K_\varepsilon$ be the number of zeros, $1$-points 
and poles of $\phi$ in $\{z\colon \varepsilon/2<|z|<2/\varepsilon\}$.
For large $k$ we then have 
\begin{equation} \label{b7a}
n(r_k/\varepsilon) - n(\varepsilon r_k)\leq K_\varepsilon .
\end{equation}
For $\varepsilon\leq t\leq 1/\varepsilon$ and large $k$ it follows that 
\begin{equation} \label{b7}
n(t r_k) \leq n(r_k/\varepsilon)
\leq n(\varepsilon r_k)+K_\varepsilon
\leq n(r_k)+K_\varepsilon
\leq (1+\varepsilon) n(r_k)
\end{equation}
as well as
\begin{equation} \label{b7b}
n(t r_k) \geq n(\varepsilon r_k)
\geq n(r_k/\varepsilon )-K_\varepsilon
\geq n(r_k)-K_\varepsilon
\geq (1-\varepsilon) n(r_k) .
\end{equation}
Thus $(r_k)$ is a sequence of P\'olya peaks for $n(r)$ 
of order~$0$ of both the first and second kind.

To prove $(b)$,
we note that $f$ has order at least~$1$
and thus is not normal by Ostrowski's result~\cite{Ostrowski1925}.
Hence there exists a sequence $(r_k)$  such that $(f_k)$ is not normal in  $\C\setminus\{0\}$.
We will proceed as in the proof of Theorem~\ref{theorem1}, but
this time $S_0$ will be the sector in $\C$ which is opposite to~$L_0$, 
and not its intersection with the annulus~$A$.
Similarly, $S_1$ and $S_\infty$ are sectors in~$\C$, and so are the sectors
$T_a$, $S_a'$ and $S_a''$ with $a\in\{0,1,\infty\}$.
For example, $T_1:=S_0\cup S_\infty\cup L_1\setminus\{0\}$.
As the rays $L_0$, $L_1$ and $L_\infty$
are equally spaced, the sectors $S_0$, $S_1$ and $S_\infty$ have opening
angles $2\pi/3$.

As in the proof of Theorem~\ref{theorem1} we define
$u_k$, $v_k$ and $w_k$ by~\eqref{a10}, \eqref{a12} and~\eqref{a15}.
Passing to a subsequence of $(r_k)$ if necessary we find as
in the proof of Theorem~\ref{theorem1} that these sequences
converge in the appropriate sectors; that is,
we have~\eqref{a11}, \eqref{a13} and~\eqref{a16}.
With $h$ defined by~\eqref{a18} we find again that~\eqref{a19} holds. 

Lemma~\ref{lemma12} yields that $u$ has the form
$u(z)=\re (e^{i\tau}(az^{3/2}+b/z^{3/2}))$ where $a,b,\tau\in\R$ with $a,b\geq 0$.
Since
\begin{equation} \label{a10a}
u_k(\xi/r_k)= \frac{\log|f_k(\xi/r_k)|}{\log|f_k(\xi/r_k)|}
= \frac{\log|f(\xi)|}{\log|f(r_k \xi)|} \to 0
\end{equation}
we deduce that $b=0$.
This implies that $h$ has the form
\begin{equation} \label{a10b}
h(z)=\re (cz^3)
\end{equation}
for some $c\in\C\setminus\{0\}$.

It follows from~\eqref{a11} and~\eqref{a10b} 
there exists a sequence $(c_k)$ in $\C$ such that
\begin{equation} \label{a11a}
\log f(r_k z)\sim c_k z^{3/2}
\quad\text{for}\ z\in T_1.
\end{equation}
Now $f(r_k z)=1$ if and only if $\log f(r_k z)=2\pi i m$ for some $m\in \Z$.
This implies that if $0<\delta<\varepsilon<1$, then
\begin{equation} \label{a11b}
n(t r_k , 1) - n(\delta r_k, 1) \sim \frac{|c_k|}{2\pi} \left( t^{3/2}-\delta^{3/2}\right) 
\quad\text{for}\  \varepsilon \leq t\leq\frac{1}{\varepsilon}.
\end{equation}
Putting 
\begin{equation} \label{a11c}
a_k=  n(\delta r_k, 1)-\frac{|c_k|\delta^{3/2}}{2\pi} 
\quad\text{and}\quad
b_k=\frac{|c_k|}{2\pi} 
\end{equation}
we find that there exists a sequence $(\varepsilon_k)$ tending to $0$ such that 
\begin{equation} \label{d1}
n(t r_k,1) - a_k \sim b_k t^{3/2}
\quad\text{for}\  \varepsilon_k \leq t\leq\frac{1}{\varepsilon_k}.
\end{equation}
The same reasoning can be made for zeros and poles and this yields that 
\begin{equation} \label{d2}
n(t r_k) - A_k \sim B_k t^{3/2}
\quad\text{for}\  \varepsilon_k \leq t\leq\frac{1}{\varepsilon_k}
\end{equation}
for suitable $A_k,B_k\in\R$ with $B_k>0$.
Noting that $n(\varepsilon_k r_k)\geq 0$ we deduce from~\eqref{d2} that 
\begin{equation} \label{d2a}
A_k\geq -(1+o(1))B_k\varepsilon_k^{3/2}.
\end{equation}
Together with~\eqref{d2} this implies that 
if $1 \leq t\leq 1/\varepsilon_k$, then
\begin{equation} \label{d2b}
\begin{aligned}
2t^{3/2} n(r_k)-n(tr_k) 
&= 2 t^{3/2} \left( A_k +(1+o(1)) B_k\right) - A_k -(1+o(1)) t^{3/2} B_k
\\
&= (2 t^{3/2} -1) A_k +(1+o(1)) t^{3/2} B_k
\\ &
\geq \left(
- (1+o(1))( 2 t^{3/2} -1) \varepsilon_k^{3/2}
+(1+o(1)) t^{3/2}
\right)B_k
\\ &
\geq 0 
\end{aligned}
\end{equation}
for large~$k$.
Thus
\begin{equation} \label{d2c}
\frac{n(tr_k)}{t^{3/2} n(r_k)}\leq 2
\quad\text{for}\  1 \leq t\leq\frac{1}{\varepsilon_k}
\end{equation}
for large~$k$.
This implies that $\rho_*\leq 3/2$.

To prove $(c)$, let $(r_k)$ be a sequence of P\'olya peaks (of the first kind)
for $n(r)$ of order $\lambda>0$. It follows from $(a)$ that 
$(f_k)$
is not normal. Thus we may assume that~\eqref{d2} holds.

Let $M>1>\varepsilon>0$. By the definition of P\'olya peaks we have 
\begin{equation} \label{d3}
n(\varepsilon r_k)\leq (1+\varepsilon) \varepsilon^\lambda n(r_k) ,
\end{equation}
for large~$k$.  Together with~\eqref{d2} this yields that
\begin{equation} \label{d4}
\begin{aligned}
(1-\varepsilon) B_k \varepsilon^{3/2}
&\leq n(\varepsilon r_k) - A_k
\\ &
\leq (1+\varepsilon)\varepsilon^\lambda n(r_k) - A_k
\\ &
\leq (1+\varepsilon)\varepsilon^\lambda \left( A_k+(1+\varepsilon)
B_k \right) -A_k .
\end{aligned}
\end{equation}
Hence
\begin{equation} \label{d5}
\left(1- (1+\varepsilon)\varepsilon^\lambda \right)A_k 
\leq \left((1+\varepsilon)^2 \varepsilon^\lambda -(1-\varepsilon)\varepsilon^{3/2}\right) B_k.
\end{equation}
Similarly,
\begin{equation} \label{d6}
\left(1- (1+\varepsilon)M^\lambda \right)A_k 
\leq \left((1+\varepsilon)^2 M^\lambda -(1-\varepsilon)M^{3/2}\right) B_k .
\end{equation}
The last two inequalities imply  that 
\begin{equation} \label{d7}
\frac{(1+\varepsilon)^2 \varepsilon^\lambda -(1-\varepsilon)\varepsilon^{3/2}}{1- (1+\varepsilon)\varepsilon^\lambda}
\geq
\frac{A_k}{B_k}
\geq
\frac{(1-\varepsilon)M^{3/2}-(1+\varepsilon)^2 M^\lambda }{(1+\varepsilon)M^\lambda-1} .
\end{equation}
Suppose now that $\lambda<3/2$. Then for small $\varepsilon$ the left hand side
is less than~$1$, while for large $M$ the right hand side is greater than~$1$.
This is a contradiction. This implies that there are no P\'olya peaks of 
the first kind of order less than~$3/2$.

The same arguments arguments can be made for P\'olya peaks of the second kind.
This yields there are no  P\'olya peaks of the second kind of order greater than~$3/2$.
Lemma~\ref{lemma6} yields that $\rho^*=\rho_*=3/2$, 
meaning that all P\'olya peaks of the first or second kind have order $3/2$.
This completes the proof of~$(c)$.

As explained above, it follows from $(a)$, $(b)$ and $(c)$ 
that if $(r_k)$ tends to~$\infty$, then $(f_k)$ is not normal  in $\C\setminus\{0\}$.
Moreover, $f$ and $n(r)$ have order~$3/2$. 

Next we show that $f$ has only finitely many critical points; that is,
$f'$ has only finitely many zeros and $f$ has only finitely many multiple poles.
Suppose that  $f$ has infinitely many critical points.
Then one of the sectors $T_0$, $T_1$ and $T_\infty$ contains a closed subsector
which contains infinitely many critical points. 
Without loss of generality we may assume that this holds for~$T_1$; 
 say $(z_k)$ is a sequence of critical points contained in a closed
subsector $T_1'$ of $T_1$ such that $r_k:=|z_k|\to\infty$.
As the sequence $(f_k)$ is not normal, we may assume that~\eqref{a11a} holds.
Differentiating we obtain 
\begin{equation} \label{d8}
\frac{r_k f'(r_k z)}{f(r_k z)}\sim \frac32c_k z^{1/2}
\quad\text{for}\ z\in T_1.
\end{equation}
This contradicts the assumption that $T_1'$ contains a critical point of 
modulus~$r_k$.
Hence $f$ has only finitely many critical points. This implies that 
the Schwarzian $S(f)$ has only finitely many poles so that $N(r,S(f))=\O(\log r)$.

Since $f$ has finite order, the lemma on the logarithmic
derivative (see~\cite[Section~3.1]{Goldberg2008} or~\cite[Section~2.2]{Hayman1964})
yields that $m(r,S(f))=\O(\log r)$. It follows that 
\begin{equation} \label{d9}
T(r,S(f))=N(r,S(f))+m(r,S(f))=\O(\log r)
\end{equation}
so that $S(f)$ is rational.

Let $Q:=S(f)$. 
Since $f$ has order $3/2$, Lemma~\ref{lemma9} yields that there exists 
$a\in\C\setminus\{0\}$ such that $Q(z)\sim a z$ as $z\to\infty$.
With out loss of generality we may assume that $a$ is negative, say
$a=-c$ with $c>0$.

Lemma~\ref{lemma11} implies that the set $\{L^1,L^2,L^3\}$ of rays 
considered there coincides with the set $\{L_0,L_1,L_\infty\}$.
As $L^2$ is the negative real axis, Lemma~\ref{lemma8} implies that $Q$ is real.

Let $\omega=e^{2\pi i/3}$ and put $f_1(z):=f(\omega z)$.
Then $S(f_1)(z)=\omega^2 Q(\omega z)$.
Lemma~\ref{lemma8} implies that $\omega^2 Q(\omega z)$ is also real.

Writing 
\begin{equation} \label{d15}
Q(z)=-cz+\sum_{j=-\infty}^0 c_jz^j
\end{equation}
we have
\begin{equation} \label{d14}
\omega^2 Q(\omega z)=-cz+\sum_{j=-\infty}^0 c_j \omega^{2+j}z^j .
\end{equation}
It follows that both $c_j$ and $c_j \omega^{2+j}$ are real for all $j\leq 0$.
This implies that $c_j=0$ if $c\neq 1 \pmod 3$.
Hence $Q$ has the form $Q(z)=-z R(z^3)$ where $R(\infty)=c>0$.
Thus $f$ satisfies~\eqref{i3}.

It remains to prove the converse direction. Thus suppose 
that $R$ is a real rational functions
satisfying $0<R(\infty)<\infty$ and that~\eqref{i3} has a meromorphic solution. 
Then, as remarked after Lemma~\ref{lemma11},
the equation~\eqref{i3} also has a meromorphic solution $f$ with the 
asymptotic values $0$, $1$ and~$\infty$.

Without loss of generality we may assume that $e^{3\theta i}=-1$.
Putting $Q(z):=-zR(z^3)$ we thus have $S(f)=Q$.
In view of Lemma~\ref{lemma11} we may assume without loss of generality 
that all but finitely many $1$-points of $f$ are contained in a small 
sector bisected by $L^2=(-\infty,0]$.

The functions $\overline{f(\overline{z})}$ and $1/f(z)$ have the same asymptotic
values in the sectors~$V_j$.
Since both functions have Schwarzian derivative~$Q$, and thus by Lemma~\ref{lemma11}
differ only by a linear fractional transformation, 
this yields that they are actually equal; that is,
\begin{equation} \label{d16}
\frac{1}{f(z)}=\overline{f(\overline{z})}.
\end{equation}
It follows from~\eqref{d16} that
the $1$-points of $f$ are symmetric with respect to the real axis.

We may write $f=w_1/w_2$ where the $w_j$ satisfy $w_j''+Aw_j=0$ with $A=Q/2$.
We have $f=1$ if and only if $w:=w_1-w_2=0$.
Thus the zeros of $w$ are also symmetric with respect to the real axis.
This implies that $\overline{w(\overline{z})}=cw(z)$ where $c=e^{i\gamma}$
for some~$\gamma\in\R$. Thus $u:=e^{i\gamma/2}w$ is real on the real axis.
Choosing $\alpha<\pi/3$ we deduce from Lemma~\ref{lemma14} and
Remark~\ref{remark1} all but finitely many zeros of $u$ are negative.

It follows that all but finitely many $1$-points are contained in the negative real 
axis~$L^2$. The proof that the other two rays $L^1$ and $L^3$ contain all but finitely
many 
zeros and poles follows with the same argument.
\end{proof}
\begin{remark}\label{Rpoles}
The main objective of the papers of Nevanlinna~\cite{Nevanlinna1932} and 
Elf\-ving~\cite{Elfving1934} cited above
was to study Riemann surfaces with finitely many branch points. They showed that
such surfaces correspond to meromorphic functions with rational Schwarzian derivative.

Elfving described such surfaces (and functions) in terms of \emph{line complexes}
(also called \emph{Speiser graphs}).
We do not give the definition of a line complex here, but refer
to~\cite[Section~ 7.4]{Goldberg2008} and~\cite[Section XI.2]{Nevanlinna1953}.
Two line complexes are sketched in Figure~\ref{line-complex}.
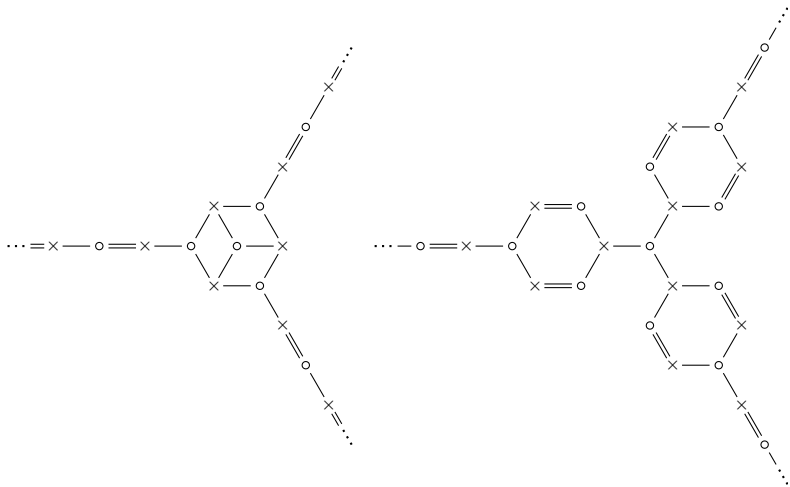
\begin{figure}[!htb]
\captionsetup{width=.85\textwidth}
\centering
\begin{tikzpicture}[scale=0.61,>=latex](-5,-5)(5,5)
\clip (-5,-5.2) rectangle (12.0,5.2);
\draw (0,0) circle (0.08);
\draw (-1,0) circle (0.08);
\draw (-3,0) circle (0.08);
\draw[-] (-2+0.2,0) -- (-2+0.8,0);
\draw[-] (-3+0.2,0.04) -- (-3+0.8,0.04);
\draw[-] (-3+0.2,-0.04) -- (-3+0.8,-0.04);
\draw[-] (-4+0.2,0) -- (-4+0.8,0);
\draw[-] (-5+0.5,0.04) -- (-5+0.8,0.04);
\draw[-] (-5+0.5,-0.04) -- (-5+0.8,-0.04);
\draw[-,thick,dotted] (-5,0) -- (-5+0.43,0);
\node at (-2,0) {\tiny $\times$};
\node at (-4,0) {\tiny $\times$};
\node at (-0.5,0.866) {\tiny $\times$};
\node at (-0.5,-0.866) {\tiny $\times$};
\draw[-] (-1+0.2*0.5,0.2*0.866) -- (-1+0.8*0.5,0.8*0.866);
\draw[-] (-1+0.2*0.5,-0.2*0.866) -- (-1+0.8*0.5,-0.8*0.866);
\draw[-] (-0.2*0.5,0.2*0.866) -- (-0.8*0.5,0.8*0.866);
\begin{scope}[shift={(0,0)},rotate=120]
\draw (-1,0) circle (0.08);
\draw (-3,0) circle (0.08);
\draw[-] (-2+0.2,0) -- (-2+0.8,0);
\draw[-] (-3+0.2,0.04) -- (-3+0.8,0.04);
\draw[-] (-3+0.2,-0.04) -- (-3+0.8,-0.04);
\draw[-] (-4+0.2,0) -- (-4+0.8,0);
\draw[-] (-5+0.5,0.04) -- (-5+0.8,0.04);
\draw[-] (-5+0.5,-0.04) -- (-5+0.8,-0.04);
\draw[-,thick,dotted] (-5,0) -- (-5+0.43,0);
\draw[-] (-1+0.2*0.5,0.2*0.866) -- (-1+0.8*0.5,0.8*0.866);
\draw[-] (-1+0.2*0.5,-0.2*0.866) -- (-1+0.8*0.5,-0.8*0.866);
\draw[-] (-0.2*0.5,0.2*0.866) -- (-0.8*0.5,0.8*0.866);
\end{scope}
\begin{scope}[shift={(0,0)},rotate=240]
\draw (-1,0) circle (0.08);
\draw (-3,0) circle (0.08);
\draw[-] (-2+0.2,0) -- (-2+0.8,0);
\draw[-] (-3+0.2,0.04) -- (-3+0.8,0.04);
\draw[-] (-3+0.2,-0.04) -- (-3+0.8,-0.04);
\draw[-] (-4+0.2,0) -- (-4+0.8,0);
\draw[-] (-5+0.5,0.04) -- (-5+0.8,0.04);
\draw[-] (-5+0.5,-0.04) -- (-5+0.8,-0.04);
\draw[-,thick,dotted] (-5,0) -- (-5+0.43,0);
\draw[-] (-1+0.2*0.5,0.2*0.866) -- (-1+0.8*0.5,0.8*0.866);
\draw[-] (-1+0.2*0.5,-0.2*0.866) -- (-1+0.8*0.5,-0.8*0.866);
\draw[-] (-0.2*0.5,0.2*0.866) -- (-0.8*0.5,0.8*0.866);
\end{scope}
\node at (1,0) {\tiny $\times$};
\node at (4*0.5,4*0.866) {\tiny $\times$};
\node at (4*0.5,-4*0.866) {\tiny $\times$};
\node at (1,1.732) {\tiny $\times$};
\node at (1,-1.732) {\tiny $\times$};
\draw (9-0,0) circle (0.08);
\node at (9-2.5,0.866) {\tiny $\times$};
\node at (9-2.5,-0.866) {\tiny $\times$};
\node at (9-1,0) {\tiny $\times$};
\node at (9-4,0) {\tiny $\times$};
\begin{scope}[shift={(9,0)},rotate=0]
\draw (-3,0) circle (0.08);
\draw (-5,0) circle (0.08);
\draw (-1.5,0.866) circle (0.08);
\draw (-1.5,-0.866) circle (0.08);
\draw[-] (-1-0.2*0.5,0.2*0.866) -- (-1-0.8*0.5,0.8*0.866);
\draw[-] (-3+0.2*0.5,0.2*0.866) -- (-3+0.8*0.5,0.8*0.866);
\draw[-] (-1-0.2*0.5,-0.2*0.866) -- (-1-0.8*0.5,-0.8*0.866);
\draw[-] (-3+0.2*0.5,-0.2*0.866) -- (-3+0.8*0.5,-0.8*0.866);
\draw[-] (-1+0.2,0) -- (-1+0.8,0);
\draw[-] (-5+0.2,0.04) -- (-5+0.8,0.04);
\draw[-] (-5+0.2,-0.04) -- (-5+0.8,-0.04);
\draw[-] (-4+0.2,0) -- (-4+0.8,0);
\draw[-] (-6+0.5,0) -- (-6+0.8,0);
\draw[-,thick,dotted] (-6,0) -- (-6+0.43,0);
\draw[-] (-2.5+0.2,0.866+0.04) -- (-2.5+0.8,0.866+0.04);
\draw[-] (-2.5+0.2,0.866-0.04) -- (-2.5+0.8,0.866-0.04);
\draw[-] (-2.5+0.2,-0.866+0.04) -- (-2.5+0.8,-0.866+0.04);
\draw[-] (-2.5+0.2,-0.866-0.04) -- (-2.5+0.8,-0.866-0.04);
\end{scope}
\begin{scope}[shift={(9,0)},rotate=120]
\draw (-3,0) circle (0.08);
\draw (-5,0) circle (0.08);
\draw (-1.5,0.866) circle (0.08);
\draw (-1.5,-0.866) circle (0.08);
\draw[-] (-1-0.2*0.5,0.2*0.866) -- (-1-0.8*0.5,0.8*0.866);
\draw[-] (-3+0.2*0.5,0.2*0.866) -- (-3+0.8*0.5,0.8*0.866);
\draw[-] (-1-0.2*0.5,-0.2*0.866) -- (-1-0.8*0.5,-0.8*0.866);
\draw[-] (-3+0.2*0.5,-0.2*0.866) -- (-3+0.8*0.5,-0.8*0.866);
\draw[-] (-1+0.2,0) -- (-1+0.8,0);
\draw[-] (-5+0.2,0.04) -- (-5+0.8,0.04);
\draw[-] (-5+0.2,-0.04) -- (-5+0.8,-0.04);
\draw[-] (-4+0.2,0) -- (-4+0.8,0);
\draw[-] (-6+0.5,0) -- (-6+0.8,0);
\draw[-,thick,dotted] (-6,0) -- (-6+0.43,0);
\draw[-] (-2.5+0.2,0.866+0.04) -- (-2.5+0.8,0.866+0.04);
\draw[-] (-2.5+0.2,0.866-0.04) -- (-2.5+0.8,0.866-0.04);
\draw[-] (-2.5+0.2,-0.866+0.04) -- (-2.5+0.8,-0.866+0.04);
\draw[-] (-2.5+0.2,-0.866-0.04) -- (-2.5+0.8,-0.866-0.04);
\end{scope}
\begin{scope}[shift={(9,0)},rotate=240]
\draw (-3,0) circle (0.08);
\draw (-5,0) circle (0.08);
\draw (-1.5,0.866) circle (0.08);
\draw (-1.5,-0.866) circle (0.08);
\draw[-] (-1-0.2*0.5,0.2*0.866) -- (-1-0.8*0.5,0.8*0.866);
\draw[-] (-3+0.2*0.5,0.2*0.866) -- (-3+0.8*0.5,0.8*0.866);
\draw[-] (-1-0.2*0.5,-0.2*0.866) -- (-1-0.8*0.5,-0.8*0.866);
\draw[-] (-3+0.2*0.5,-0.2*0.866) -- (-3+0.8*0.5,-0.8*0.866);
\draw[-] (-1+0.2,0) -- (-1+0.8,0);
\draw[-] (-5+0.2,0.04) -- (-5+0.8,0.04);
\draw[-] (-5+0.2,-0.04) -- (-5+0.8,-0.04);
\draw[-] (-4+0.2,0) -- (-4+0.8,0);
\draw[-] (-6+0.5,0) -- (-6+0.8,0);
\draw[-,thick,dotted] (-6,0) -- (-6+0.43,0);
\draw[-] (-2.5+0.2,0.866+0.04) -- (-2.5+0.8,0.866+0.04);
\draw[-] (-2.5+0.2,0.866-0.04) -- (-2.5+0.8,0.866-0.04);
\draw[-] (-2.5+0.2,-0.866+0.04) -- (-2.5+0.8,-0.866+0.04);
\draw[-] (-2.5+0.2,-0.866-0.04) -- (-2.5+0.8,-0.866-0.04);
\end{scope}
\node at (9+0.5,-3*0.866) {\tiny $\times$};
\node at (9+0.5,+3*0.866) {\tiny $\times$};
\node at (9+0.5,0.866) {\tiny $\times$};
\node at (9+0.5,-0.866) {\tiny $\times$};
\node at (9+2,3.464) {\tiny $\times$};
\node at (9+2,-3.464) {\tiny $\times$};
\node at (9+2,-1.732) {\tiny $\times$};
\node at (9+2,1.732) {\tiny $\times$};
\end{tikzpicture}
\caption{Two line complexes.}
\label{line-complex}
\end{figure}
The left one was also considered by Elfving~\cite[Section~2, Abb.~3]{Elfving1934}.
The function corresponding to this line complex 
has three logarithmic singularities and three critical points,
and the critical values corresponding to these three critical points
coincide with the three logarithmic singularities.

Elfving~\cite[Section~47]{Elfving1934} 
considered how symmetry of the line complex is reflected in the function;
see also~\cite[Section~42]{Nevanlinna1932}.
For the line complexes given in Figure~\ref{line-complex}, and the associated meromorphic
functions~$f$, it follows~\cite[p.~59]{Elfving1934} that 
$S(f)$ has the form~\eqref{i3} with rational functions~$R$
satisfying $R(\infty)\in\C\setminus\{0\}$.
In addition, the mirror symmetry of the line complexes implies that $R$ is real.

For the left line complex in Figure~\ref{line-complex},
the function $f$ has only three (simple) critical points.
Hence $S(f)$ has three (double) poles. Thus $R$ has only one (double) pole $p$ and hence
the form 
\begin{equation} \label{d17}
R(z) =-c +\frac{a}{z-p}+\frac{b}{(z-p)^2}.
\end{equation}
We recall that Elfving~\cite[Kapitel IV]{Elfving1934} determined 
for which rational functions $Q$ the equation $S(f)=Q$ has a meromorphic solution~$f$.
It can be deduced from his result that if $R$ is given by~\eqref{d17},
then~\eqref{i3} has a meromorphic solution
if and only if $b=-27p/2$ and $c=(4a^2+36a+45)/72p$.

We may assume that $-c=R(\infty)<0$ and that $f$ has logarithmic singularities over
$0$, $1$ and~$\infty$, with the $1$-points close to the negative real axis,
corresponding to the branch of the line complex which extends to the left.
The simple $1$-points then correspond to the double edges of the line complex on
this branch, and there is one double $1$-point corresponding to the diamond 
at the end of this branch.
Since $1$-points are symmetric with respect to the real axis,
it follows that all $1$-points must lie on the negative real axis.

Thus there are rational functions $R$ with poles such that~\eqref{i3} has a solution $f$ 
for which all (and not only all but finitely many) zeros, $1$-points and poles lie on three rays.

For the right line complex in Figure~\ref{line-complex} the situation is different.
Assume again that the $1$-points are distributed along the negative real axis,
corresponding to the branch of the line complex which extends to the left.
The center of the hexagon on this branch corresponds to a negative $1$-point.
However, there are also further $1$-points corresponding to double
edges of the hexagons on the other branches. So it may happen that not all but only 
all but finitely many zeros, $1$-points and poles lie on the rays.

Putting more than one hexagon on the branches stretching to~$\infty$, or replacing the
hexagons by $(4n+2)$-gons for some $n>1$, we find that
the rational function $R$ in~\eqref{i3} may have arbitrarily high degree.
\end{remark}
\begin{remark}\label{remark2}
In the proof of Theorem~\ref{theorem3}, we have used Lemma~\ref{lemma14}
to prove that zeros, $1$-points and poles are on the respective rays.
Alternatively, we could have used
the symmetry of the associated line complex, similarly to the 
reasoning in Remark~\ref{Rpoles}.
\end{remark}

\noindent
%Walter Bergweiler\\
Mathematisches Seminar\\
Christian-Albrechts-Universit\"at zu Kiel\\
Ludewig-Meyn-Str.\ 4\\
24098 Kiel\\
Germany\\
{\tt Email: bergweiler@math.uni-kiel.de}

\medskip

\noindent
%Alexandre Eremenko\\
Department of Mathematics\\
Purdue University\\
West Lafayette, IN 47907\\
USA\\
{\tt Email: eremenko@math.purdue.edu}

\end{document}